\documentclass{birkjour}
\usepackage{amsthm}
\usepackage{amsfonts}
\usepackage{mathrsfs,yhmath}
\usepackage{amscd,xy,amssymb,amsmath,graphicx,verbatim}
\usepackage{lipsum}
\usepackage[TS1,OT1,T1]{fontenc}
\newtheorem{theorem}{Theorem}[section]
\newtheorem{lemma}[theorem]{Lemma}

\newtheorem{proposition}[theorem]{Proposition}

\theoremstyle{definition}
\newtheorem{definition}[theorem]{Definition}

\theoremstyle{remark}

\numberwithin{equation}{section}

\begin{document}

%
%
%
%
%
%
%
%
%

\title[{A note on connectedness of Blaschke products}]{A note on connectedness of Blaschke products}

\author[Yue Xin]{Yue Xin}
\address{School of Mathematics, Jilin University, 130012, Changchun, P. R. China}
\email{179929393@qq.com}
\author[Bingzhe Hou]{Bingzhe Hou}

\address{School of Mathematics, Jilin University, 130012, Changchun, P. R. China}
\email{houbz@jlu.edu.cn}



\begin{abstract}
Consider the space $\mathcal{F}$ of all inner functions on the unit open disk under the uniform topology, which is a metric topology induced by the $H^{\infty}$-norm. In the present paper, a class of Blaschke products, denoted by $\mathcal{H}_{SC}$, is introduced. We prove that for each $B\in\mathcal{H}_{SC}$, $B$ and $zB$ belong to the same path-connected component of $\mathcal{F}$. It plays an important role of a method to select a fine subsequence of zeros. As a byproduct, we obtain that each Blaschke product in $\mathcal{H}_{SC}$ has an interpolating and one-component factor.
\end{abstract}
\maketitle

\section{Introduction}

Let $\mathbb{D}$ be the unit open disk in the complex plane $\mathbb{C}$ and let $\partial\mathbb{D}$ be the boundary of $\mathbb{D}$, i.e., the unit circle.
The pseudo-hyperbolic distance on the unit open disk $\mathbb{D}$, denoted by $\rho$, is given by
\begin{equation*}
	\rho(z,w)={\big \vert}\frac{z-w}{1-\overline{z}w}{\big \vert}, \ \ \ \text{for any} \ z, w \in\mathbb{D}.
\end{equation*}
Let $H^{\infty}$ be the Banach algebra of bounded analytic functions on $\mathbb{D}$ equipped with the norm ${\Vert}f{\Vert}_{\infty}=\sup_{z\in\mathbb{D}}{\vert}f(z){\vert}$. A bounded analytic function $f$ on $\mathbb{D}$ is called an inner function if it has unimodular radial limits almost everywhere on the boundary $\partial\mathbb{D}$ of $\mathbb{D}$. Furthermore, denote by $\mathcal{F}$ the set of all inner functions. We are interested in the space $\mathcal{F}$ under uniform topology, which is a metric topology induced by the $H^{\infty}$-norm. Notice that the uniform topology on $\mathcal{F}$ is very complicated and interesting, see \cite{He74, Ne79, Ne80} for instance.

A Blaschke product is an inner function of the form
\begin{equation*}
	B(z)=\lambda z^{m}\prod_{n}\frac{{\vert}z_{n}{\vert}}{z_{n}}\frac{z_{n}-z}{1-\overline{z_{n}}z},
\end{equation*}
where $m$ is a nonnegative integer, $\lambda$ is a complex number with ${\vert}\lambda{\vert}=1$, and $\{z_{n}\}$ is a sequence of points in $\mathbb{D}\setminus\{0\}$ satisfying the Blaschke condition $\sum_{n}(1-{\vert}z_{n}{\vert})<\infty$. Moreover, if $\lambda=1$, we say that $B$ is normalized.

If for every bounded sequence of complex numbers $\{w_n\}_{n=1}^{\infty}$, there exists $f$ in $H^{\infty}$ satisfying $f(z_n)=w_n$ for every $n\in\mathbb{N}$, then both the sequence $\{z_n\}_{n=1}^{\infty}$ and the Blaschke product $B(z)$ are called interpolating. Following from a celebrated result of Carleson \cite{LC}, one can see
that $B(z)$ is an interpolating Blaschke product if and only if $\{z_{n}\}$ is a uniformly separated sequence, i.e.,
\begin{equation*}
	\inf_{n\in \mathbb{N}}\prod_{k\neq n}{\big \vert}\frac{z_{k}-z_{n}}{1-\overline{z_{k}}z_{n}}{\big \vert}>0.
\end{equation*}
Moreover, if
\begin{equation*}
	\lim_{n\rightarrow\infty}\prod_{k\neq n}{\big \vert}\frac{z_{k}-z_{n}}{1-\overline{z_{k}}z_{n}}{\big \vert}=1,
\end{equation*}
both the sequence $\{z_n\}_{n=1}^{\infty}$ and the Blaschke product $B(z)$ are called thin.
In addition, A Blaschke product is called Carleson-Newman if it is a product of finitely many interpolating Blaschke products.
Interpolating Blaschke products and Carleson-Newman Blaschke products play an important role in the study of $H^{\infty}$. As well-known, inner functions can be approximated uniformly by Blaschke products, and Carleson-Newman Blaschke products can be approximated uniformly by interpolating Blaschke products \cite{DEA}. However, there is still an open problem whether the set of all interpolation Blaschke products is dense in the inner function space $\mathcal{F}$. There has been obtained some related results. For instance, Marshall (\cite{DE}) proved that finite linear combinations of Blaschke products are dense in $H^{\infty}$; Nicolau and Su\'{a}rez \cite{NA} characterize the connected components of the subset $CN^{*}$ of $H^{\infty}$ formed by the products $bh$, where $b$ is a Carleson-Newman Blaschke product and $h\in H^{\infty}$ is an invertible function.

In particular, a result of K. Tse \cite{KF} tells us that a sequence $\{z_n\}_{n=1}^{\infty}$ of points contained in a Stolz domain
$$
\{z\in\mathbb{D}: {\vert}1-\overline{\xi}z{\vert}\le C(1-{\vert}z{\vert})\},
$$
where $\xi$ is a constant with ${\vert}\xi{\vert}=1$, is interpolating if and only if it is
separated, i.e.,
\begin{equation*}
	\inf_{m\neq n}\rho(z_{m}, z_{n})>0.
\end{equation*}

Then, if the zero points lie in a Stolz domain, we may say something more about the Blaschke products. For example,
A. Reijonen gave a sufficient condition for a Blaschke product with zeros in a Stolz domain to be a one-component inner function in \cite{Re19}. An inner function $u$ in $H^{\infty}$ is said to be one-component if there is $\eta\in(0,1)$ such that the level set $\Omega_{u}(\eta) :=\{z\in\mathbb{D}:{\vert}u(z){\vert}<\eta\}$ is connected. More details of one-component inner functions, we refer to \cite{AB, Jo, Re19}.

In this paper, we focus on the path-connected components of the space $\mathcal{F}$ under the uniform topology. Recall that the topology of the uniform convergence on the set $\mathcal{F}$ is induced by the following metric
\begin{equation*}
d(f,g)={\Vert}f-g{\Vert}_{\infty}=\sup_{z\in\mathbb{D}}{\vert}f(z)-g(z){\vert}=\sup_{\theta\in \mathbb{R}}{\rm ess}{\vert}f(e^{{\bf i}\theta})-g(e^{{\bf i}\theta}){\vert}.
\end{equation*}
For any two inner functions $f$ and $g$, if they belong to the same path-connected component in the space $\mathcal{F}$, we denote $f\sim g$.

The path of inner functions has always been of great interest and is related to many important issues.
D. Herrero \cite{He74} considered the path-connected components of the space $\mathcal{F}$. He showed that a component of $\mathcal{F}$ can contain nothing but Blaschke products with infinitely many zeroes, exactly one (up to a constant factor) singular inner function or infinitely many pairwise coprime singular inner functions, which answered a problem of Douglas. Furthermore, V. Nestoridis studied the invariant and noninvariant connected components of the inner functions space $\mathcal{F}$. He proved that the inner functions $d(z)={\rm exp} \{(z+1)/(z-1)\}$ and $zd$ belong to the same connected component \cite{Ne79}, and gave a family of inner functions, denoted by $H$, such that for every $B\in H$, $B$ and $zB$ don't belong to the same component \cite{Ne80}. In particular, the family $H$ contains only Blaschke products, and contains all of thin Blaschke products.
In addition, several authors have studied the connected component of inner functions in the context of model spaces and operator theory (see \cite{Al16, BC} for instance).

In the present paper, we aim to give a class of Blaschke products, denoted by $\mathcal{H}_{SC}$, such that each $B\in\mathcal{H}_{SC}$, $B$ and $zB$ belong to the same connected component. Firstly, let us define a class of subsets of the unit open disk, named strip cones and denoted by $SC(\xi,\theta,T_{1}, T_{2})$. In this paper, we study the Blaschke products with zeros lying in a strip cone.

\begin{definition}\label{SC}
Let $\theta\in(0,\pi)$, $\xi\in \partial\mathbb{D}$, $T_1$ and $T_2$ be two nonzero real numbers. Denote by $J_{i}$ the arc on the circle
$$
{\vert}z-(1-T_{i}e^{{\bf i}\theta})\xi{\vert}={\vert}T_i{\vert}
$$
in the unit open disk, for $i=1,2$. Write $\xi_i$ as the intersection point of $J_{i}$ and $\partial\mathbb{D}$ other than $\xi$, $i=1,2$.  We define $SC(\xi,\theta,T_{1}, T_{2})$ is the region bounded by $J_{1}$, $J_{2}$ and $\wideparen{\xi_1\xi_2}$ which is the arc on the unit circle $\partial\mathbb{D}$ from $\xi_1$ to $\xi_2$ without $\xi$, and call it a strip cone. If $T_1=T_2\in \mathbb{R}\setminus\{0\}$, then $SC(\xi,\theta,T_{1}, T_{2})$ is just the arc $J_{1}=J_{2}$. In particular, if $T_1$ and $T_2$ are infinity, then $SC(\xi,\theta,T_{1}, T_{2})$ is just the segment $J_{1}=J_{2}=(-1,1)$.
\end{definition}

One can see some examples of strip cones in Figure \ref{figure1}. Next, we explain why we name it strip cone.

\begin{figure}[htbp]
\centering
\includegraphics[width=0.7\textwidth]{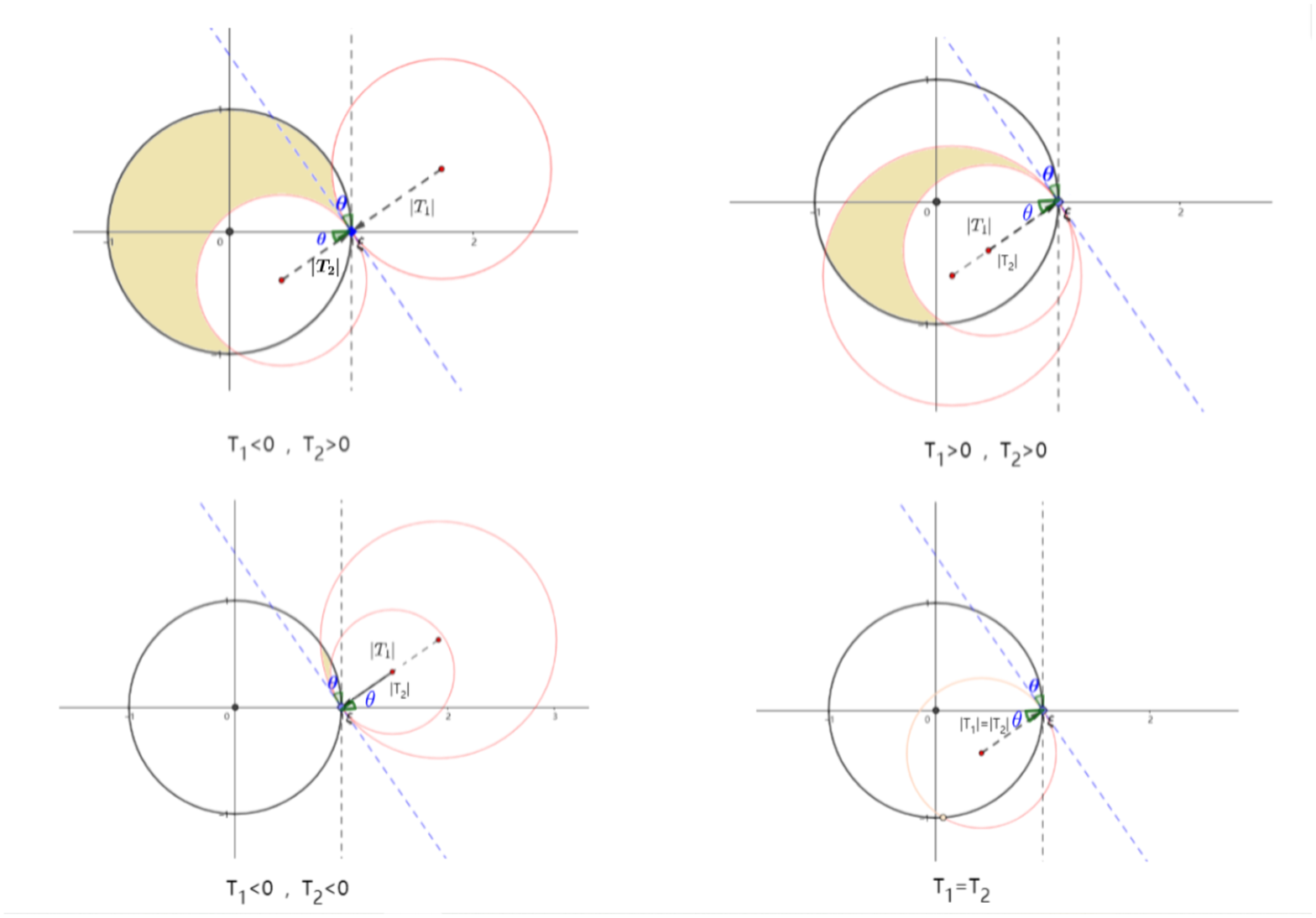}
\caption{Examples of strip cones}
\label{figure1}
\end{figure}

\begin{definition}
Let $\theta\in(-\frac{\pi}{2},\frac{\pi}{2})$, and $L_{1}$ and $L_{2}$ be two parallel straight lines with an angle of $\theta$ to the real axis. Denote by $SL(\theta,L_{1},L_{2})$ the strip region between $L_{1}$ and $L_{2}$ in the right half plane.
\end{definition}

Given any strip cone $SC(\xi,\theta,T_{1}, T_{2})$. Consider the fractional linear transformation $\varphi_{\xi}(z)=(\xi+z)/(\xi-z)$, where ${\vert}\xi{\vert}=1$. It is easy to see that $\varphi_{\xi}$ maps the unit open disk onto the right half plane, and map the arcs $J_1$ and $J_2$ in Definition \ref{SC} to some parallel straight lines $L_{1}$ and $L_{2}$ with the angle of $\theta$ to the imaginary axis in the right half plane. Then, $\varphi_{\xi}$ is an analytic bijection from $SC(\xi,\theta,T_{1},T_{2})$ to $SL(\frac{\pi}{2}-\theta,L_{1},L_{2})$ (see Figure \ref{figure2} for instance). This is the reason we call the subset $SC(\xi,\theta,T_{1},T_{2})$ strip cone. Without loss of generality, we may assume that $\xi=1$, because it is no difference to deal with $\xi=1$ and general $\xi$ with ${\vert}\xi{\vert}=1$. Moreover, we always denote $\varphi(z)=(1+z)/(1-z)$ through this paper.
\begin{figure}[htbp]
	\centering
	\includegraphics[width=0.6\textwidth]{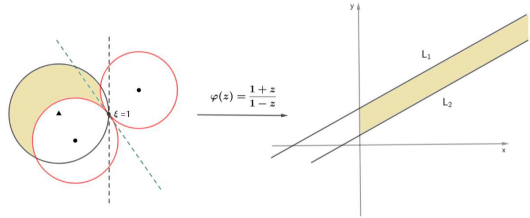}
	\caption{$\varphi(z)$ maps $SC(1,\theta,T_{1},T_{2})$ onto $SL(\frac{\pi}{2}-\theta,L_{1},L_{2})$}
	\label{figure2}
\end{figure}
\begin{definition}
Denote by $\mathcal{H}_{SC}$ the family of all Blaschke products $B$ satisfying the following conditions,
\begin{enumerate}
\item[(i)]the zeros $\{z_{n}\}_{n=1}^{\infty}$ of $B$ lie in some certain strip cone $SC(\xi,\theta,T_{1},T_{2})$;
\item[(ii)] ${\vert}\xi-z_{n}{\vert}$ non-increasingly tends to $0$;
\item[(iii)]there exists a positive number $\delta<1$, such that $\rho(z_{n},z_{n+1})\le\delta$ for any $n\in\mathbb{N}$.
\end{enumerate}
\end{definition}

Now we show our main result as the following theorem.

\vspace{0.25cm}
${\bf Main\ \  Theorem} \ \ $ For any $B\in\mathcal{H}_{SC}$, $B$ and $zB$ belong to the same path-connected component of the inner functions space $\mathcal{F}$ under the uniform topology.
\vspace{0.25cm}

In the next section, we will introduce a method to select a "fine" factor of a Blaschke product in $\mathcal{H}_{SC}$, which plays an important role to prove the main theorem. As a byproduct, we obtain that each Blaschke product in $\mathcal{H}_{SC}$ has an interpolating and one-component factor. Then, we will complete the proof of the main theorem in the last section.

\section{Preliminaries}

First of all,  let us start from the following simple result.

\begin{lemma}\label{factorconnect}
Let $f=\varphi_{1}\cdot\varphi_{2}$ and $g=\psi_{1}\cdot\psi_{2}$, where $f,g,\varphi_{1},\varphi_{2},\psi_{1},\psi_{2}\in\mathcal{F}$. If $\varphi_{1}\sim \psi_{1}$ and $\varphi_{2}\sim \psi_{2}$, then $f\sim g$. In particular, if $\varphi_{1}\sim z\varphi_{1}$, then $f\sim zf$.
\end{lemma}

To prove ${B}\sim z{B}$, it suffices to prove $\widetilde{B}\sim z\widetilde{B}$, if $\widetilde{B}$ is factor of $B$. In this section, it will be shown that we can select a factor $\widetilde{B}$ of $B$ such that $\widetilde{B}$ satisfies more conditions than $B$. The major is how to choose a subsequence of the zeros sequence of $B$. Recall that for a Blaschke product $B\in \mathcal{H}_{SC}$ with zeros $\{z_{n}\}_{n=1}^{\infty}$, there exists a positive number $\delta<1$, such that $\rho(z_{n},z_{n+1})\le\delta$ for any $n\in\mathbb{N}$.

\begin{lemma}\label{le2}
Let $a$, $a{'}$, $b$, $b{'}$ be real numbers satisfying
\begin{equation*}
0\leq a\leq a{'}<1\ \ \text{and} \ \ 0\leq b\leq b{'}<1.
\end{equation*}
Then,
\begin{equation*}
\frac{a+b}{1+{a}b}\leq\frac{a{'}+b{'}}{1+{a{'}}b{'}}
\end{equation*}
\end{lemma}

\begin{proof}
\begin{align*}
\frac{a{'}+b{'}}{1+{a{'}}b{'}}-\frac{a+b}{1+{a}b}&=\frac{(a{'}+b{'}+a{'}ab+abb{'})-(a+b+aa{'}b{'}+a{'}bb{'})}{(1+ab)(1+a{'}b{'})}\\
&=\frac{(a{'}-a)(1-bb{'})+(b{'}-b)(1-a{'}a)}{(1+ab)(1+a{'}b{'})}\\
&\geq 0.
\end{align*}
\end{proof}

\begin{lemma}\label{sublu}
Let $\{z_{n}\}_{n=1}^{\infty}$ be a sequence of complex numbers in the unit open disk, satisfying that
\begin{enumerate}
\item for any $m\in\mathbb{N}$, $\rho(z_{m},z_{n})\rightarrow 1$ as $n\rightarrow\infty$;
\item there exists a positive number $0<\delta<1$, such that $\rho(z_{n},z_{n+1})\le\delta$ for any $n\in\mathbb{N}$.
\end{enumerate}
Then, for any $0<\varepsilon<1$, we can choose a subsequence $\{z_{n_{k}}\}_{n_{k}=1}^{\infty}$ of $\{z_{n}\}_{n=1}^{\infty}$ such that
\begin{equation*}
0<\varepsilon\le\rho(z_{n_{k}},z_{n_{k+1}})\le\frac{\varepsilon+\delta}{1+\varepsilon\delta}<1.
\end{equation*}
\end{lemma}

\begin{proof}
Given any $0<\varepsilon<1$. Put $z_{n_{1}}=z_{1}$. Since $\rho(z_{n_{1}},z_{i})\rightarrow 1$ as $i\rightarrow\infty$, we can choose
$$
n_{2}=\min\{i; \ \rho(z_{n_{1}},z_{i})\ge\varepsilon\}.
$$
It is obvious that $\rho(z_{n_{1}},z_{n_{2}})\geq\varepsilon$.
	
With the same method, we choose ${n_{k+1}}$ by
$$
{n_{k+1}}=\min\{{i}; \ i> n_k, \ \rho(z_{n_{k}},z_{i})\ge\varepsilon\}.
$$
Then, $\rho(z_{n_{k}},z_{n_{k+1}})\geq\varepsilon$. Moreover, for each $k=1,2,\ldots$,
$$
0\leq\rho(z_{n_{k}},z_{n_{k+1}-1})<\varepsilon<1 \ \ \text{and} \ \ 0\leq\rho(z_{n_{k+1}-1},z_{n_{k+1}})<\delta<1.
$$
By Lemma \ref{le2}, we have
\begin{equation*}	
\frac{\rho(z_{n_{k}},z_{n_{k+1}-1})+\rho(z_{n_{k+1}-1},z_{n_{k+1}})}{1+\rho(z_{n_{k}},z_{n_{k+1}-1})\rho(z_{n_{k+1}-1},z_{n_{k+1}})}
\le\frac{\varepsilon+\delta}{1+\varepsilon\delta}.
\end{equation*}
Therefore,		
$$			
0<\varepsilon\le\rho(z_{n_{k}},z_{n_{k+1}})\le\frac{\rho(z_{n_{k}},z_{n_{k+1}-1})+\rho(z_{n_{k+1}-1},z_{n_{k+1}})}{1+\rho(z_{n_{k}},z_{n_{k+1}-1})\rho(z_{n_{k+1}-1},z_{n_{k+1}})}
\le\frac{\varepsilon+\delta}{1+\varepsilon\delta}<1.
$$
\end{proof}

\begin{lemma}\label{ab}
Let $\alpha$ and $\beta$ be two complex numbers in the unit open disk with ${\vert}1-\alpha{\vert}<{\vert}1-\beta{\vert}$. Denote $\alpha=s_{1}+{\bf i}(1-s_{1})\cot\theta_{1}$ and $\beta=s_{2}+{\bf i}(1-s_{2})\cot\theta_{2}$, where $\theta_1, \theta_2\in(0, \pi)$. For the pair of $\alpha$ and $\beta$, let the positive numbers $\varepsilon$, $\delta$, $\theta_0$, $C$, $\tau$ and $\eta$ satisfy the following conditions,
\begin{enumerate}
\item[(1)] \ $0<\varepsilon\le\rho(\alpha,\beta)\le\delta<1$;
\item[(2)] \ $\theta_0\in(0, \pi)$ and $C={\vert}1-{\bf i}\cot\theta_{0}{\vert}$;
\item[(3)] \ ${\vert}\cot\theta_{i}-\cot\theta_{0}{\vert}<\tau<\sqrt{\frac{3 C^2\varepsilon^{2}}{16 C^2(1-\varepsilon^{2})+3\varepsilon^{2}}}$, for $i=1,2$;
\item[(4)] \ $3\leq 4-2\eta\leq(1+s_{1}-(1-s_{1})\cot^{2}\theta_{1})(1+s_{2}-(1-s_{2})\cot^{2}\theta_{2})\leq 4$.
\end{enumerate}
Then
\begin{equation*}
0<C_1\le{\big \vert}\frac{1-\alpha}{1-\beta}{\big \vert}\le C_2 <1.
\end{equation*}
where
{\small
\begin{align*}
&C_1=\frac{C-\tau}{C+\tau}\left(\frac{C-\tau}{C+\tau}+\frac{2\delta^{2}-2\sqrt{\delta^{4}+\delta^{2}(C^2-\tau^2)(1-\delta^{2})}}{(C+\tau)^{2}(1-\delta^{2})}\right),
\\
&C_2=\frac{C+\tau}{C-\tau}\left(\frac{C+\tau}{C-\tau}+\frac{(2-\eta)\varepsilon^{2}-\sqrt{(2-\eta)^2\varepsilon^{4}+2\varepsilon^{2}(2-\eta)(C^2-\tau^2)(1-\varepsilon^{2})}}{(C-\tau)^{2}(1-\varepsilon^{2})} \right).
\end{align*}
}
\end{lemma}

\begin{proof}
For $i=1,2$, it follows from condition $(3)$ that		
$$
{\vert}(1-{\bf i}\cot\theta_{i})-(1-{\bf i}\cot\theta_{0}){\vert}<\tau,
$$
and consequently,
$$
0<C-\tau<{\vert}1-{\bf i}\cot\theta_{i}{\vert}<C+\tau.
$$
Denote
$$
K=\frac{1-s_{1}}{1-s_{2}}.
$$
Notice that
\begin{equation*}
\begin{aligned}
&0<(1-s_{1})(C-\tau)\le{\vert}1-\alpha{\vert}=(1-s_{1}){\vert}1-{\bf i}\cot\theta_{1}{\vert}\le(1-s_{1})(C+\tau),\\
&0<(1-s_{2})(C-\tau)\le{\vert}1-\beta{\vert}=(1-s_{2}){\vert}1-{\bf i}\cot\theta_{2}{\vert}\le(1-s_{2})(C+\tau).
\end{aligned}
\end{equation*}
Then
\begin{equation}\label{eq7}
\frac{K(C-\tau)}{C+\tau}\le{\big \vert}\frac{1-\alpha}{1-\beta}{\big \vert}\le \frac{K(C+\tau)}{C-\tau}.
\end{equation}
Since
$$
\rho(\alpha,\beta)^{2}=\frac{{\vert}\alpha-\beta{\vert}^{2}}{(1-\overline{\alpha}\beta)(1-\overline{\beta}\alpha)}=\frac{1}{\frac{(1-{\vert}\alpha{\vert}^{2})(1-{\vert}\beta{\vert}^{2})}{{\vert}\alpha-\beta{\vert}^{2}}+1},
$$
it follows from condition $(1)$ that 
\begin{equation}\label{eq:3}
\begin{aligned}
0<\frac{1}{\delta^{2}}-1\le\frac{(1-{\vert}\alpha{\vert}^{2})(1-{\vert}\beta{\vert}^{2})}{{\vert}\alpha-\beta{\vert}^{2}}\le\frac{1}{\varepsilon^{2}}-1.
\end{aligned}
\end{equation}

Consider $\frac{(1-{\vert}\alpha{\vert}^{2})(1-{\vert}\beta{\vert}^{2})}{{\vert}\alpha-\beta{\vert}^{2}}$. We have
\begin{equation*}
\begin{aligned}
&\frac{(1-{\vert}\alpha{\vert}^{2})(1-{\vert}\beta{\vert}^{2})}{{\vert}\alpha-\beta{\vert}^{2}} \\
=&\frac{(1-(s_{1}^{2}+(1-s_{1})^{2}\cot^{2}\theta_{1}))\cdot(1-(s_{2}^{2}+(1-s_{2})^{2}\cot^{2}\theta_{2}))}{{\vert}(1-\beta)-(1-\alpha){\vert}^{2}}\\
=&\frac{(1-s_{1})(1-s_{2})(1+s_{1}-(1-s_{1})\cot^{2}\theta_{1})(1+s_{2}-(1-s_{2})\cot^{2}\theta_{2}))}{{\vert}(1-s_{1})(1-{\bf i}\cot\theta_{1})-(1-s_{2})(1-{\bf i}\cot\theta_{2}){\vert}^{2}}.
\end{aligned}
\end{equation*}
Denote
$$
{W}_{1}=1+s_{1}-(1-s_{1})\cot^{2}\theta_{1} \ \ \text{and}  \ \ {W}_{2}=1+s_{2}-(1-s_{2})\cot^{2}\theta_{2}.
$$
Then,
$$
3\leq 4-2\eta\leq{W}_{1}\cdot{W}_{2}\leq4.
$$

Now we give the lower bound and upper bound of $K$, respectively. Then, by inequality (\ref{eq7}), we can complete the proof.

{\bf Lower bound of $K$.}  \ If $K\leq \frac{C-\tau}{C+\tau}$, we have
\begin{equation*}
\begin{aligned}
\frac{(1-{\vert}\alpha{\vert}^{2})(1-{\vert}\beta{\vert}^{2})}{{\vert}\alpha-\beta{\vert}^{2}}=&\frac{(1-s_{1})(1-s_{2}){W}_{1}{W}_{2}}{{\vert}(1-s_{1})(1-{\bf i}\cot\theta_{1})-(1-s_{2})(1-{\bf i}\cot\theta_{2}){\vert}^{2}}\\
=&\frac{K{W}_{1}{W}_{2}}{{\vert}K(1-{\bf i}\cot\theta_{1})-(1-{\bf i}\cot\theta_{2}){\vert}^{2}}\\
\le&\frac{4K}{({\vert}(1-{\bf i}\cot\theta_{2}){\vert}-{\vert}K(1-{\bf i}\cot\theta_{1}){\vert})^{2}}\\
\le&\frac{4K}{((C-\tau)-K(C+\tau))^{2}}.
\end{aligned}
\end{equation*}
It follows from inequality (\ref{eq:3}) that
\begin{equation*}
\frac{4K}{((C-\tau)-K(C+\tau))^{2}}\ge\frac{(1-{\vert}\alpha{\vert}^{2})(1-{\vert}\beta{\vert}^{2})}{{\vert}\alpha-\beta{\vert}^{2}}\ge\frac1{\delta^{2}}-1.
\end{equation*}
Then,
$$
K^{2}-2\left(\frac{C-\tau}{C+\tau}+\frac{2\delta^{2}}{(C+\tau)^{2}(1-\delta^{2})}\right)K+\frac{(C-\tau)^2}{(C+\tau)^2}\le0,
$$
Let $a_1$ and $a_2$ be the two roots of the above quadratic polynomial of $K$,
$$
a_1=\frac{C-\tau}{C+\tau}+\frac{2\delta^{2}-2\sqrt{\delta^{4}+\delta^{2}(C^2-\tau^2)(1-\delta^{2})}}{(C+\tau)^{2}(1-\delta^{2})},
$$
$$
a_2=\frac{C-\tau}{C+\tau}+\frac{2\delta^{2}+2\sqrt{\delta^{4}+\delta^{2}(C^2-\tau^2)(1-\delta^{2})}}{(C+\tau)^{2}(1-\delta^{2})}.
$$
One can see that
$$
0<a_1<\frac{C-\tau}{C+\tau}<a_2.
$$
Then, we always have
$$
K\geq a_1>0.
$$
Moreover, put $C_1=\frac{C-\tau}{C+\tau}\cdot a_1$,
$$
{\big \vert}\frac{1-\alpha}{1-\beta}{\big \vert}\ge\frac{K(C-\tau)}{C+\tau}\geq C_1>0.
$$

{\bf Upper bound of $K$.}  Notice that
\begin{equation*}
\begin{aligned}
&\frac{(1-{\vert}\alpha{\vert}^{2})(1-{\vert}\beta{\vert}^{2})}{{\vert}\alpha-\beta{\vert}^{2}} \\
=&\frac{K{W}_{1}{W}_{2}}{{\vert}K(1-{\bf i}\cot\theta_{1})-(1-{\bf i}\cot\theta_{2}){\vert}^{2}}\\
=&\frac{K{W}_{1}{W}_{2}}{{\big \vert}iK(\cot\theta_{0}-\cot\theta_{1})-i(\cot\theta_{0}-\cot\theta_{2})+(K-1)(1-{\bf i}\cot\theta_{0}){\big \vert}^{2}}\\
\ge&\frac{(4-2\eta)K}{\left({\big \vert}K(\cot\theta_{0}-\cot\theta_{1}){\big \vert}+{\big \vert}\cot\theta_{0}-\cot\theta_{2}{\big \vert}+{\big \vert}(K-1)C{\big \vert}\right)^{2}}\\
\ge&\frac{(4-2\eta)K}{({\vert}(K-1){\vert}C+(K+1)\tau)^{2}}.
\end{aligned}
\end{equation*}
Firstly, let us prove that $K$ must be less than $1$.
If $K\geq 1$, we have
$$
\frac{(1-{\vert}\alpha{\vert}^{2})(1-{\vert}\beta{\vert}^{2})}{{\vert}\alpha-\beta{\vert}^{2}}\ge\frac{(4-2\eta)K}{((1-K)C+(K+1)\tau)^{2}}.
$$
It follows from inequality (\ref{eq:3}) that
\begin{equation*}
\frac{(4-2\eta)K}{((K-1)C+(K+1)\tau)^{2}}\le\frac{(1-{\vert}\alpha{\vert}^{2})(1-{\vert}\beta{\vert}^{2})}{{\vert}\alpha-\beta{\vert}^{2}}\le\frac1{\varepsilon^{2}}-1.
\end{equation*}
Then,
\begin{equation*}
K^{2}-2\left(\frac{C-\tau}{C+\tau}+\frac{(2-\eta)\varepsilon^{2}}{(1-\varepsilon^{2})(C+\tau)^{2}}\right)K+\frac{(C-\tau)^2}{(C+\tau)^2}\ge0,
\end{equation*}
Let $\lambda_1$ and $\lambda_2$ be the two roots of the above quadratic polynomial of $K$,
$$
\lambda_1=\frac{C-\tau}{C+\tau}+\frac{(2-\eta)\varepsilon^{2}-\sqrt{(2-\eta)^2\varepsilon^{4}+2\varepsilon^{2}(2-\eta)(C^2-\tau^2)(1-\varepsilon^{2})}}{(C+\tau)^{2}(1-\varepsilon^{2})},
$$
$$
\lambda_2=\frac{C-\tau}{C+\tau}+\frac{(2-\eta)\varepsilon^{2}+\sqrt{(2-\eta)^2\varepsilon^{4}+2\varepsilon^{2}(2-\eta)(C^2-\tau^2)(1-\varepsilon^{2})}}{(C+\tau)^{2}(1-\varepsilon^{2})}.
$$
Since
\begin{align*}
&\left(\frac{C+\tau}{C-\tau} \right)^{2}-2\left(\frac{C-\tau}{C+\tau}+\frac{(2-\eta)\varepsilon^{2}}{(1-\varepsilon^{2})(C+\tau)^{2}}\right)\left(\frac{C+\tau}{C-\tau} \right)+\frac{(C-\tau)^2}{(C+\tau)^2} \\
= \ &\left(\frac{C+\tau}{C-\tau}-\frac{C-\tau}{C+\tau} \right)^{2}-\frac{(4-2\eta)\varepsilon^{2}}{(1-\varepsilon^{2})(C+\tau)^{2}}\cdot\left(\frac{C+\tau}{C-\tau} \right) \\
= \ &\frac{16C^2\tau^2}{(C^2-\tau^2)^2}-\frac{(4-2\eta)\varepsilon^{2}}{(1-\varepsilon^{2})(C^2-\tau^2)} \\
= \ &\frac{16C^2}{(C^2-\tau^2)^2}\left(\tau^2-\frac{3\varepsilon^{2}(C^2-\tau^2)}{16 C^2(1-\varepsilon^{2})}\right) \\
< \ &0,
\end{align*}
one can see that
$$
0<\lambda_1<\frac{C-\tau}{C+\tau}<1<\frac{C+\tau}{C-\tau}<\lambda_2.
$$
Then, we have
$$
K\geq \lambda_2>\frac{C+\tau}{C-\tau}>1.
$$
However, by inequality (\ref{eq7}),
$$
{\big \vert}\frac{1-\alpha}{1-\beta}{\big \vert}\ge\frac{K(C-\tau)}{C+\tau}\ge \lambda_2\cdot \frac{C-\tau}{C+\tau}>1.
$$
It is a contradiction to ${\vert}1-\alpha{\vert}<{\vert}1-\beta{\vert}$.

Now we have known $K<1$. Then,
$$
\frac{(1-{\vert}\alpha{\vert}^{2})(1-{\vert}\beta{\vert}^{2})}{{\vert}\alpha-\beta{\vert}^{2}}\ge\frac{(4-2\eta)K}{((1-K)C+(K+1)\tau)^{2}}.
$$
It follows from inequality (\ref{eq:3}) that
\begin{equation*}
\frac{(4-2\eta)K}{((1-K)C+(K+1)\tau)^{2}}\le\frac{(1-{\vert}\alpha{\vert}^{2})(1-{\vert}\beta{\vert}^{2})}{{\vert}\alpha-\beta{\vert}^{2}}\le\frac1{\varepsilon^{2}}-1.
\end{equation*}
Then,
\begin{equation*}
K^{2}-2\left(\frac{C+\tau}{C-\tau}+\frac{(2-\eta)\varepsilon^{2}}{(1-\varepsilon^{2})(C-\tau)^{2}}\right)K+\frac{(C+\tau)^2}{(C-\tau)^2}\ge0.
\end{equation*}
Let $A_1$ and $A_2$ be the two roots of the above quadratic polynomial of $K$,
$$
A_1=\frac{C+\tau}{C-\tau}+\frac{(2-\eta)\varepsilon^{2}-\sqrt{(2-\eta)^2\varepsilon^{4}+2\varepsilon^{2}(2-\eta)(C^2-\tau^2)(1-\varepsilon^{2})}}{(C-\tau)^{2}(1-\varepsilon^{2})},
$$
$$
A_2=\frac{C+\tau}{C-\tau}+\frac{(2-\eta)\varepsilon^{2}+\sqrt{(2-\eta)^2\varepsilon^{4}+2\varepsilon^{2}(2-\eta)(C^2-\tau^2)(1-\varepsilon^{2})}}{(C-\tau)^{2}(1-\varepsilon^{2})}.
$$
Since
\begin{align*}
&\left(\frac{C-\tau}{C+\tau} \right)^{2}-2\left(\frac{C+\tau}{C-\tau}+\frac{(2-\eta)\varepsilon^{2}}{(1-\varepsilon^{2})(C-\tau)^{2}}\right)\left(\frac{C-\tau}{C+\tau} \right)+\frac{(C+\tau)^2}{(C-\tau)^2} \\
= \ &\left(\frac{C-\tau}{C+\tau}-\frac{C+\tau}{C-\tau} \right)^{2}-\frac{(4-2\eta)\varepsilon^{2}}{(1-\varepsilon^{2})(C-\tau)^{2}}\cdot\left(\frac{C-\tau}{C+\tau} \right) \\
= \ &\frac{16C^2\tau^2}{(C^2-\tau^2)^2}-\frac{(4-2\eta)\varepsilon^{2}}{(1-\varepsilon^{2})(C^2-\tau^2)} \\
= \ &\frac{16C^2}{(C^2-\tau^2)^2}\left(\tau^2-\frac{3\varepsilon^{2}(C^2-\tau^2)}{16 C^2(1-\varepsilon^{2})}\right) \\
< \ &0,
\end{align*}
one can see that
$$
0<A_1<\frac{C-\tau}{C+\tau}<1<A_2.
$$
Then, we always have
$$
K\leq A_1<\frac{C-\tau}{C+\tau}<1.
$$
Moreover, put $C_2=\frac{C+\tau}{C-\tau}\cdot A_1$,
$$
{\big \vert}\frac{1-\alpha}{1-\beta}{\big \vert}\le\frac{K(C+\tau)}{C-\tau}\leq C_2<1.
$$
\end{proof}

\begin{lemma}\label{genlu}
Let $\{z_{n}\}_{n=1}^{\infty}$ be a sequence of complex numbers in a strip cone $SC(\xi,\theta_0,T_{1},T_{2})$, $\theta_0\in (0, \pi)$, satisfying that
\begin{enumerate}
\item ${\vert}1-z_{n}{\vert}$ non-increasingly tends to $1$, as $n\rightarrow\infty$;
\item there exist two positive numbers $0<\varepsilon\le\delta<1$, such that for any $n\in\mathbb{N}$,
\[
0<\varepsilon\le \rho(z_{n},z_{n+1})\le\delta<1.
\]
\end{enumerate}	
Then there exists a positive integer $N$ and two positive constants $C_{1}$ and $C_{2}$, such that for any $n\geq N$,
\begin{equation*}
0<C_{1}\le{\big \vert}\frac{1-z_{n+1}}{1-z_{n}}{\big \vert}\le C_{2}<1.
\end{equation*}
Furthermore, this implies $\sum_{n=1}^{\infty}{\vert}1-z_n{\vert}<\infty$.
\end{lemma}
\begin{proof}
Without loss of generality, we may assume that $\{z_{n}\}_{n=1}^{\infty}$ lie in a strip cone $SC(1,\theta_0,T_{1},T_{2})$ (see Figure \ref{figure3}).
\begin{figure}[htbp]
		\centering
		\includegraphics[width=0.8\textwidth]{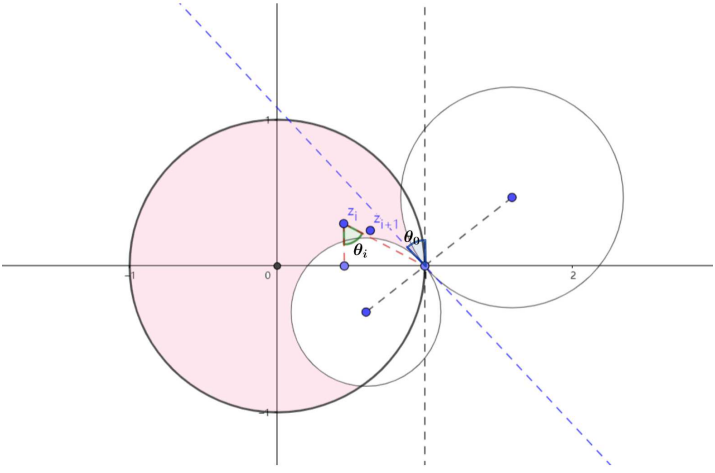}
		\caption{Zeros $\{z_{n}\}_{n=1}^{\infty}$ of $B(z)$}
		\label{figure3}
	\end{figure}	

Denote
$$
z_{n}=x_{n}+{\bf i}(1-x_{n})\cot\theta_{n}, \ \ \ \ \text{for} \ n=1,2,\ldots,
$$
and
$$
C={\vert}1-{\bf i}\cot\theta_{0}{\vert}.
$$
Following from $z_n\in SC(1,\theta_0,T_{1},T_{2})$,
we could also denote
$$
z_{n}=(1-t_{n}{\rm e}^{{\bf i}\theta_{0}})+t_n{\rm e}^{{\bf i}\zeta_{n}},
$$
where $t_n\in [T_1, T_2]$ if $T_1T_2>0$, and $t_n\in (-\infty, T_1]\cup [T_2, +\infty)$ if $T_1T_2<0$.
Furthermore,
$$
\lim\limits_{n\rightarrow\infty}1-z_n=\lim\limits_{n\rightarrow\infty}t_n({\rm e}^{{\bf i}\theta_{0}}-{\rm e}^{{\bf i}\zeta_{n}})=0=\lim\limits_{n\rightarrow\infty}1-x_n.
$$
Notice that $t_n$ is uniformly far away from $0$ whenever $T_1T_2>0$ or $T_1T_2<0$. Then,
$\zeta_{n}\rightarrow\theta_{0}$, as $n\rightarrow\infty$. Consequently,
$$
\lim\limits_{n\rightarrow\infty}\cot\theta_{n}=\lim\limits_{n\rightarrow\infty}\frac{{\rm Imz_n}}{1-{\rm Rez_n}}
=\lim\limits_{n\rightarrow\infty}\frac{t_n(\sin\zeta_n-\sin\theta_0)}{t_n(\cos\zeta_n-\cos\theta_0)}=\cot\theta_0.
$$
Hence, for any two positive number
$$
0<\eta\le\frac{1}{2} \ \ \ \text{and} \ \ \ 0<\tau\le\sqrt{\frac{3 C^2\varepsilon^{2}}{16 C^2(1-\varepsilon^{2})+3\varepsilon^{2}}},
$$
there exists $N\in \mathbb{N}$ such that, for all $n\geq N$,
$$
3\leq 4-2\eta\leq(1+x_{n}-(1-x_{n})\cot^{2}\theta_{n})(1+x_{n+1}-(1-x_{n+1})\cot^{2}\theta_{n+1})\leq 4
$$
and
$$
{\vert}\cot\theta_{n}-\cot\theta_{0}{\vert}<\tau.
$$
Therefore, by Lemma \ref{ab}, there exist two positive constants $C_{1}$ and $C_{2}$, such that for any $n\geq N$,
\begin{equation*}
0<C_{1}\le{\big \vert}\frac{1-z_{n+1}}{1-z_{n}}{\big \vert}\le C_{2}<1.
\end{equation*}
Furthermore, this implies
\[
\sum\limits_{n=1}^{\infty}{\vert}1-z_n{\vert}\le \sum_{n=1}^{\infty}{\vert}1-z_1{\vert}C_2^{n-1}=\frac{{\vert}1-z_1{\vert}}{1-C_2}<\infty.
\]
\end{proof}

Now, we show that one can select a "fine" subsequence of zeros of each Blaschke product $B\in\mathcal{H}_{SC}$.

\begin{definition}\label{fine}
A sequence $\{z_{n}\}_{n=1}^{\infty}$ is said to be fine, if it satisfies the following conditions.
\begin{enumerate}
\item[(1)] \ The sequence $\{z_{n}\}_{n=1}^{\infty}$ lies in some certain strip cone $SC(\xi,\theta, T_{1},T_{2})$.
\item[(2)] \ ${\vert}\xi-z_{n}{\vert}$ non-increasingly tends to $\xi$, as $n\rightarrow\infty$;
\item[(3)] \ There exist two positive numbers $0<\varepsilon\le\delta<1$ such that for any $n\in\mathbb{N}$,
\[
0<\varepsilon\le\rho(z_{n},z_{n+1})\le\delta.
\]
Moreover, there exists a positive integer $N$ and two positive constants $C_{1}$ and $C_{2}$, such that for any $n\geq N$,
\begin{equation*}
0<C_{1}\le{\big \vert}\frac{\xi-z_{n+1}}{\xi-z_{n}}{\big \vert}\le C_{2}<1.
\end{equation*}
In particular, $\sum\limits_{n=1}^{\infty}{\vert}\xi-z_{n}{\vert}<\infty$.
\item[(4)] \  ${\rm Re} \varphi_{\xi}(z_{n})$ monotonically tends to $+\infty$, where $\varphi_{\xi}(z)=(\xi+z)/(\xi-z)$. Moreover, there are two positive numbers $\widetilde{C_1}$ and $\widetilde{C_2}$ such that
$$
0<\widetilde{C_1}\leq \frac{{\rm Re}\varphi_{\xi}(z_{n})}{{\rm Re}\varphi_{\xi}(z_{n+1})} \leq \widetilde{C_2}<1.
$$
\end{enumerate}
Furthermore, a Blaschke product is said to be fine if its zeros sequence $\{z_{n}\}_{n=1}^{\infty}$ is fine. And denote by $\mathcal{\widetilde{H}}_{SC}$ the family of all fine Blaschke products.
\end{definition}

\begin{lemma}\label{finesub}
Let $B$ be a Blaschke product in $\mathcal{H}_{SC}$ with zeros $\{z_{n}\}_{n=1}^{\infty}$. Then $B$ has a factor in $\mathcal{\widetilde{H}}_{SC}$, i.e., the sequence $\{z_{n}\}_{n=1}^{\infty}$ has a fine subsequence.
\end{lemma}

\begin{proof}
Without loss of generality, assume $\{z_{n}\}_{n=1}^{\infty}$ lies in some certain strip cone $SC(1,\theta_0, T_{1},T_{2})$. Recall that $\{z_{n}\}_{n=1}^{\infty}$  satisfies the following conditions.
\begin{enumerate}
\item \  ${\vert}1-z_{n}{\vert}$ non-increasingly tends to $1$, as $n\rightarrow\infty$.
\item \  There exists a positive number $0<\delta<1$, such that $\rho(z_{n},z_{n+1})\le\delta$ for any $n\in\mathbb{N}$.
\end{enumerate}
Denote by $SL(\frac{\pi}{2}-\theta_0,L_{1},L_{2})$ the image of the strip cone $SC(1, \theta_0, T_{1},T_{2})$ under $\varphi$. We may write
$$
L_{1}:y-\tan(\frac{\pi}{2}-\theta_0)\cdot x-c_1=0 \ \ \  \text{and} \ \ \ L_{2}:y-\tan(\frac{\pi}{2}-\theta_0)\cdot x-c_2=0.
$$

Since ${\vert}1-z_{n+1}{\vert}\leq {\vert}1-z_n{\vert}$ for each $n\in\mathbb{N}$, we have
\begin{align*}
{\vert}\varphi(z_{n+1}){\vert}=&{\big \vert}\frac{1+z_{n+1}}{1-z_{n+1}} {\big \vert} \\
=&{\big \vert}1-\frac{2}{1-z_{n+1}} {\big \vert} \\
\geq &{\big \vert}\frac{2}{1-z_{n+1}} {\big \vert}-1 \\
\geq &{\big \vert}\frac{2}{1-z_{n}} {\big \vert}-1 \\
\geq &{\big \vert}\frac{2}{1-z_{n}}-1 {\big \vert}-2 \\
=& {\vert}\varphi(z_{n}){\vert}-2.
\end{align*}

By Lemma \ref{sublu}, there exists a subsequence $\{z_{n_{k}}\}_{k=1}^{\infty}$ of $\{z_{n}\}_{n=1}^{\infty}$ and positive numbers $\varepsilon$ and $\delta$, such that
$$
0<\varepsilon\le\rho(z_{n_{k}},z_{n_{k+1}})\le \delta<1.
$$

Consequently, by Lemma \ref{genlu}, there are  positive numbers $C_1$ and $C_2$ such that
\begin{equation*}
0<C_{1}\le{\big \vert}\frac{1-z_{n_{k+1}}}{1-z_{n_k}}{\big \vert}\le C_2<1.
\end{equation*}
Since $z_{n_k}\rightarrow 1$ as $k\rightarrow\infty$ and $\varphi(z_{n_k})\in SL(\frac{\pi}{2}-\theta_0,L_{1},L_{2})$, we have
$$
\lim\limits_{k\rightarrow\infty}{\vert}\varphi(z_{n_{k}}){\vert}=\lim\limits_{k\rightarrow\infty}{\big \vert}\frac{1+z_{n_{k}}}{1-z_{n_{k}}} {\big \vert}=+\infty,
$$
$$
\lim\limits_{k\rightarrow\infty}\frac{{\rm Re}\varphi(z_{n_{k}})}{{\vert}\varphi(z_{n_{k}}){\vert}}=\cos(\frac{\pi}{2}-\theta_0)>0
$$
and
$$
\lim\limits_{k\rightarrow\infty}{\big \vert}\frac{\varphi(z_{n_{k}})}{\varphi(z_{n_{k+1}})} {\big \vert} \bigg/  {\big \vert}\frac{1-z_{n_{k+1}}}{1-z_{n_{k}}} {\big \vert}
=\lim\limits_{k\rightarrow\infty}{\big \vert}\frac{1+z_{n_{k}}}{1+z_{n_{k+1}}} {\big \vert}=1.
$$
Furthermore,
$$
\lim\limits_{k\rightarrow\infty}\left(\frac{{\rm Re}\varphi(z_{n_{k}})}{{\rm Re}\varphi(z_{n_{k+1}})} \right) \bigg/ {\big \vert}\frac{1-z_{n_{k+1}}}{1-z_{n_{k}}} {\big \vert}=1.
$$
Then, there exists a positive integer $k_0$ such that for any $k\geq k_0$,
$$
0<\frac{C_1}{2}\leq \frac{{\rm Re}\varphi(z_{n_{k}})}{{\rm Re}\varphi(z_{n_{k+1}})} \leq \frac{1+C_2}{2}<1.
$$

Therefore, the subsequence $\{z_{n_{k}}\}_{k=k_0}^{\infty}$ is as required. In particular,
$$
{\rm Re}\varphi(z_{n_{k+1}})>{\rm Re}\varphi(z_{n_k}).
$$
\end{proof}

Then, to prove our main theorem, it suffices to consider the Blaschke products in $\widetilde{\mathcal{H}}_{SC}$.

By the way, we could also find that fine sequences implies some other properties of Blaschke products.
More precisely, we could obtain that each $B\in \mathcal{H}_{SC}$ has a factor being an interpolating and one-component Blaschke product.

\begin{lemma}[Corollary 2.5 in \cite{Jo}]\label{separ}
Let $B$ be a Blaschke product whose zeros ${z_n}$ are contained in a Stolz domain and are separated. Suppose that $\rho(z_n,z_{n+1})\le\eta<1$. Then $B$ is a one-component inner function.
\end{lemma}

\begin{theorem}\label{one-c}
Each $B\in\mathcal{H}_{SC}$ has an interpolating and one-component Blaschke product factor.
\end{theorem}
\begin{proof}
By Lemma \ref{finesub}, it suffices to prove that each $B\in\mathcal{\widetilde{H}}_{SC}$ has an interpolating and one-component Blaschke product factor. Obviously, any fine sequence has a tail contained in some certain Stolz domain, and then a Blaschke product in $\mathcal{\widetilde{H}}_{SC}$ is interpolating if and only if its zeros sequence is separated, i.e.,
\begin{equation*}
\inf\limits_{m\neq n}\rho(z_n,z_m)>0.
\end{equation*}
Furthermore, together with Lemma \ref{separ}, we only need to prove any fine sequence has a separated subsequence.

Without loss of generality, assume $\{z_{n}\}_{n=1}^{\infty}$ is a fine sequence in a strip cone $SC(1,\theta, T_{1},T_{2})$.
Since $\varphi(z)=(1+z)/(1-z)$, we have $z=(\varphi(z)-1)/(\varphi(z)+1)$, and then
\begin{align*}
	{\vert}z_m-z_n{\vert}&={\big \vert}\frac{\varphi(z_m)-1}{\varphi(z_m)+1}-\frac{\varphi(z_n)-1}{\varphi(z_n)+1}{\big \vert}\\
	&={\big \vert}\frac{2(\varphi(z_m)-\varphi(z_n))}{(\varphi(z_m)+1)(\varphi(z_n)+1)}{\big \vert}
\end{align*}
and
\begin{align*}
	{\vert}1-\overline{z_m}z_n{\vert}&={\big \vert}1-\frac{\overline{\varphi(z_m)}-1}{\overline{\varphi(z_m)}+1}\cdot\frac{\varphi(z_n)-1}{\varphi(z_n)+1}{\big \vert}\\
	&={\big \vert}\frac{2(\overline{\varphi(z_m)}+\varphi(z_n))}{(\overline{\varphi(z_m)}+1)(\varphi(z_n)+1)}{\big \vert}\\
	&={\big \vert}\frac{2[\overline{\varphi(z_m)}+\varphi(z_n)]}{(\varphi(z_m)+1)(\varphi(z_n)+1)}{\big \vert}.
\end{align*}
Therefore,
\begin{align*}
	\rho(z_m,z_n)&={\big \vert}\frac{z_m-z_n}{1-\overline{z_m}z_n}{\big \vert}\\
	&={\big \vert}\frac{\varphi(z_m)-\varphi(z_n)}{\overline{\varphi(z_m)}+\varphi(z_n)}{\big \vert}\\
	&\ge{\big \vert}\frac{{\vert}\varphi(z_m){\vert}-{\vert}\varphi(z_n){\vert}}{{\vert}\varphi(z_m){\vert}+{\vert}\varphi(z_n){\vert}}{\big \vert}\\
	&={\big \vert}\frac{1-{\vert}\frac{\varphi(z_n)}{\varphi(z_m)}{\vert}}{1+{\vert}\frac{\varphi(z_n)}{\varphi(z_m)}{\vert}}{\big \vert}.\\
\end{align*}
In addition, there are  positive numbers $C_1$ and $C_2$ such that
\begin{equation*}
	0<C_{1}\le{\big \vert}\frac{1-z_{n_{k+1}}}{1-z_{n_k}}{\big \vert}\le C_2<1.
\end{equation*}
Since ${\vert}1-z_n{\vert}$ non-increasingly tends to $1$, as $n\to\infty$, there exists a positive integer $N\in\mathbb{N}$ such that for any $n\ge N$,
\begin{equation*}
	{\big \vert}\frac{1+z_n}{1+z_{n+1}}{\big \vert}\le\frac{1+C_2}{2C_2}.
\end{equation*}
Then, we have
\begin{align*}
	{\big \vert}\frac{\varphi(z_n)}{\varphi(z_{n+1})}{\big \vert}&={\big \vert}\frac{1+z_n}{1-z_n}\cdot\frac{1-z_{n+1}}{1+z_{n+1}}{\big \vert}\\
	&={\big \vert}\frac{1+z_n}{1+z_{n+1}}{\big \vert}{\big \vert}\frac{1-z_{n+1}}{1-z_n}{\big \vert}\\
	&\le\frac{C_2 +1}{2}\\
	&<1,
\end{align*}
Consequently, for any $m>n$,
\begin{align*}
{\big \vert}\frac{\varphi(z_n)}{\varphi(z_m)}{\big \vert}\le&\left(\frac{C_2 +1}{2}\right)^{m-n} \\
\le& \frac{C_2 +1}{2} \\
<&1.
\end{align*}
Then,
\begin{align*}
\rho(z_m,z_n)&\ge{\big \vert}\frac{1-{\vert}\frac{\varphi(z_n)}{\varphi(z_m)}{\vert}}{1+{\vert}\frac{\varphi(z_n)}{\varphi(z_m)}{\vert}}{\big \vert}\\
&=1-\frac{2}{\frac{1}{{\big \vert}\frac{\varphi(z_n)}{\varphi(z_m)}{\big \vert}}+1}\\
&\ge 1-\frac{2}{\frac{1}{\frac{C_2 +1}{2}}+1}\\
&=\frac{1-C_2}{3+C_2}.
\end{align*}
Therefore, we have
\begin{equation*}
\inf\limits_{m\neq n}\rho(z_n,z_m)\ge\frac{1-C_2}{3+C_2}>0.
\end{equation*}
This finishes the proof.	
\end{proof}

\section{Proof of the main theorem}

In this section, we would prove our main theorem.
Consider a Blaschke product
$$
B=\lambda z^m\prod^{\infty}_{n=1}\frac{\overline{z_{n}}}{\mid z_{n}\mid}\cdot\frac{z_{n}-z}{1-\overline{z_{n}}z}.
$$
where ${\vert}\lambda{\vert}=1$. It is well known that the finite Blaschke products with the same order are in the same component. Then, $B\sim zB$ if and only if for some certain $N\in\mathbb{N}$,
\begin{equation}\label{NtoN1}
\prod^{\infty}_{n=N}\frac{\overline{z_{n}}}{\mid z_{n}\mid}\cdot\frac{z_{n}-z}{1-\overline{z_{n}}z} \ \sim \ \prod^{\infty}_{n=N+1}\frac{\overline{z_{n}}}{\mid z_{n}\mid}\cdot\frac{z_{n}-z}{1-\overline{z_{n}}z}.
\end{equation}

Recall that $\varphi(z)=(1+z)/(1-z)$ is the fractional linear transformation from the unit open disk to the right half plane. Let $\alpha_{n}(t)$ be the unique point in $\mathbb{D}$ such that
$$
\varphi(\alpha_{n}(t))=(1-t)\varphi(z_{n})+t\varphi(z_{n+1}),
$$
for $t\in[0,1]$ and $n=1,2,\cdots$. Furthermore, following from the idea of Nestoridis \cite{Ne79}, define the map $B_t$ from $[0,1]$ to $\mathcal{F}$ by
\begin{equation}\label{path}
t \longmapsto B_{t}=\prod^{\infty}_{n=N}\frac{\overline{\alpha_{n}(t)}}{\mid\alpha_{n}(t)\mid}\cdot\frac{\alpha_{n}(t)-z}{1-\overline{\alpha_{n}(t)}z}.
\end{equation}
If the above map is continuous, then the relation (\ref{NtoN1}) holds and consequently $B\sim zB$. To prove the continuity of $B_t$, the following theorem plays an important role.

\begin{lemma}[Lemma $1$ in \cite{Ne79}]\label{Nest}
Let
\begin{equation*}
K_{1}=\prod^{\infty}_{n=1}\frac{\overline{\alpha_{n}}}{{\vert}\alpha_{n}{\vert}}\frac{\alpha_{n}-z}{1-\overline{\alpha_{n}}z}\ \ \ \text{and}\ \ \ K_{2}=\prod^{\infty}_{n=1}\frac{\overline{\beta_{n}}}{{\vert}\beta_{n}{\vert}}\frac{\beta_{n}-z}{1-\overline{\beta_{n}}z}
\end{equation*}
be two infinite Blaschke products such that $K_{1}(0)>0$ and $K_{2}(0)>0$, then we have the following inequality,
\begin{equation}
\begin{aligned}
{\Vert}K_{1}-K_{2}{\Vert}_{\infty}\le &\sum_{n}{\big \vert}\textrm{arg}\frac{\alpha_{n}}{\beta_{n}}{\big \vert}+2\sum_{n}{\big \vert}\textrm{arg}\frac{1-\alpha_{n}}{1-\beta_{n}}{\big \vert} \\
&+2\sup_{y\in \mathbb{R}}{\rm ess}\sum_{n}{\big \vert}\textrm{arg}\frac{\varphi(\alpha_{n})-{\bf i}y}{\varphi(\beta_{n})-{\bf i}y}{\big \vert}.
\end{aligned}
\end{equation}
\end{lemma}

Now, we would apply Lemma \ref{Nest} to verify the continuity of the path $B_t$. More precisely,
\begin{equation}\label{eq:1}
	\begin{aligned}
		{\Vert}B_{t}-B_{t+\Delta t}{\Vert}_{\infty}\le \ &\sum_{n}{\big \vert}\textrm{arg}\frac{\alpha_{n}(t)}{\alpha_{n}(t+\Delta t)}{\big \vert}+2\sum_{n}{\big \vert}\textrm{arg}\frac{1-\alpha_{n}(t)}{1-\alpha_{n}(t+\Delta t)}{\big \vert}\\
		&+2\sup_{y\in \mathbb{R}}\textrm{ess}\sum_{n}{\big \vert}\textrm{arg}\frac{\varphi(\alpha_{n}(t))-{\bf i}y}{\varphi(\alpha_{n}(t+\Delta t))-{\bf i}y}{\big \vert}.
	\end{aligned}
\end{equation}
Then, to prove $\lim\limits_{\Delta t \to 0}{\Vert}B_{t}-B_{t+\Delta t}{\Vert}=0$, it suffices to prove the three items in the right of the inequality (\ref{eq:1}) tend to $0$ as $\Delta t \to 0$.

\begin{proposition}\label{prop1}
For any $B\in \mathcal{\widetilde{H}}_{SC}$, there exists a positive integer $N_1\in\mathbb{N}$ and a positive number $K_1$ such that for any $t\in [0,1]$ and small positive number $\Delta t$,
\[
\sum\limits_{n=N_1}^{\infty}{\big \vert}\textrm{arg}\frac{\alpha_{n}(t)}{\alpha_{n}(t+\Delta t)}{\big \vert}
\le K_1\Delta t.
\]
\end{proposition}

\begin{proof}
Without loss of generality, suppose that the zeros sequence of $B\in \mathcal{\widetilde{H}}_{SC}$ is fine in a strip cone $SC(1,\theta, T_1, T_2)$.
By the definition of $\alpha_{n}(t)$, we have
\begin{align*}
		1-\alpha_{n}(t)=&1-\varphi^{-1}(\varphi(\alpha_{n}(t))) \\
		=&1-\frac{\varphi(\alpha_{n}(t))-1}{\varphi(\alpha_{n}(t))+1}   \\
		=&\frac{2}{\varphi(\alpha_{n}(t))+1}
	\end{align*}
Furthermore,
\begin{align*}
&\alpha_{n}(t+\Delta t)-\alpha_{n}(t) \\
=&(1-\alpha_{n}(t))-(1-\alpha_{n}(t+\Delta t)) \\
=&\frac{2\Delta t (\varphi(z_{n+1})-\varphi(z_{n}))}{(\varphi(\alpha_{n}(t))+1)(\varphi(\alpha_{n}(t+\Delta t))+1)}  \\
=&\frac{4\Delta t (z_{n+1}-z_{n})}{(1-z_{n})(1-z_{n+1})} \cdot
  \frac{1}{(\varphi(\alpha_{n}(t))+1)(\varphi(\alpha_{n}(t+\Delta t))+1)}.
\end{align*}

Since ${\rm Re}\varphi(z_{n})$ increasingly tends to $+\infty$ and
\[
\lim\limits_{k\rightarrow\infty}\frac{{\rm Re}\varphi(z_{n})}{{\vert}\varphi(z_{n}){\vert}}=\cos(\frac{\pi}{2}-\theta)>0,
\]
there is a positive number $R_1>0$ such that
\[
{\vert}\varphi(\alpha_{n}(t))+1{\vert}\ge {\vert}{\rm Re}(\varphi(\alpha_{n}(t))+1){\vert}\ge {\vert}{\rm Re}(\varphi(z_{n})+1){\vert}\ge \frac{{\vert}\varphi(z_{n})+1{\vert}}{R_1}.
\]

In addition, the followings hold.
\begin{enumerate}
\item[(1)] \ The sequence ${\vert}1-z_{n}{\vert}$ non-increasingly tends to $1$, as $n\rightarrow\infty$.
\item[(2)] \ by Lemma \ref{genlu}, there exists a positive constants $C_{1}$ such that
\begin{equation*}
0<C_{1}\le{\big \vert}\frac{1-z_{n+1}}{1-z_{n}}{\big \vert}.
\end{equation*}
\end{enumerate}
Then, we have
\begin{align*}
&{\vert}\alpha_{n}(t+\Delta t)-\alpha_{n}(t){\vert} \\
=&{\big \vert}\frac{4\Delta t (z_{n+1}-z_{n})}{(1-z_{n})(1-z_{n+1})} {\big \vert}\cdot
  {\big \vert}\frac{1}{(\varphi(\alpha_{n}(t))+1)(\varphi(\alpha_{n}(t+\Delta t))+1)} {\big \vert} \\
\le&\frac{4\Delta t ({\vert}1-z_{n+1}{\vert}+{\vert}1-z_{n}{\vert})}{{\vert}1-z_{n}{\vert}{\vert}1-z_{n+1}{\vert}} \cdot
  \frac{R_1^2}{{\vert}\varphi(z_{n})+1{\vert}^2} \\
\le&\frac{8\Delta t {\vert}1-z_{n}{\vert}}{{\vert}1-z_{n}{\vert}{\vert}1-z_{n+1}{\vert}} \cdot
  \frac{R_1^2}{{\vert}\frac{2}{1-z_n}{\vert}^2} \\
=&2R_1^2\cdot \Delta t \cdot \frac{{\vert}1-z_{n}{\vert}}{{\vert}1-z_{n+1}{\vert}} \cdot {\vert}{1-z_n}{\vert} \\
\le& \frac{2R_1^2}{C_1}\cdot{\vert}\Delta t{\vert}\cdot {\vert}{1-z_n}{\vert}.
\end{align*}

Given a positive number $0<r<1$, it follows from $z_{n}\rightarrow 1$ that there exists a positive integer $N_1\in \mathbb{N}$  such that, for any $n\ge N_1$ and any $t\in [0,1]$,
\[
{\vert}\alpha_{n}(t){\vert}\ge r.
\]
Thus, considering the area of the triangle with vertices $0$, $\alpha_{n}(t+\Delta t)$ and $\alpha_{n}(t)$, we have
\begin{align*}
&{\big \vert}\textrm{arg}\frac{\alpha_{n}(t)}{\alpha_{n}(t+\Delta t)}{\big \vert} \\
\le& \frac{\pi}{2}\cdot \sin{\big \vert}\textrm{arg}\frac{\alpha_{n}(t)}{\alpha_{n}(t+\Delta t)}{\big \vert} \\
\le& \frac{\pi}{2}\cdot \frac{{\vert}\alpha_{n}(t+\Delta t)-\alpha_{n}(t){\vert}\cdot 1}{{\vert}\alpha_{n}(t+\Delta t){\vert}\cdot{\vert}\alpha_{n}(t){\vert}} \\
\le& \frac{R_1^2\pi}{r^2C_1}\cdot \Delta t \cdot {\vert}{1-z_n}{\vert}.
\end{align*}
Consequently, by $\sum_{n}{\vert}{1-z_n}{\vert}<\infty$, for any $t\in[0,1]$,
\[
\sum\limits_{n=N_1}^{\infty}{\big \vert}\textrm{arg}\frac{\alpha_{n}(t)}{\alpha_{n}(t+\Delta t)}{\big \vert}
\le \Delta t \cdot \frac{R_1^2\pi}{r^2C_1}\sum\limits_{n=N_1}^{\infty}{\vert}{1-z_n}{\vert}.
\]
Moreover, write
\[
K_1=\frac{R_1^2\pi}{r^2C_1}\sum\limits_{n=1}^{\infty}{\vert}{1-z_n}{\vert}
\]
as required.
\end{proof}

\begin{proposition}\label{prop2}
For any $B\in \mathcal{\widetilde{H}}_{SC}$, there exists a positive integer $N_2\in\mathbb{N}$ and a positive number $K_2$ such that for any $t\in [0,1]$ and small positive number $\Delta t$,
\[
\sup\limits_{y\in \mathbb{R}}\textrm{ess}\sum\limits_{n=N_2}^{\infty}{\big \vert}\textrm{arg}\frac{\varphi(\alpha_{n}(t))-{\bf i}y}{\varphi(\alpha_{n}(t+\Delta t))-{\bf i}y}{\big \vert}\le K_2\Delta t.
\]
\end{proposition}

\begin{proof}
Without loss of generality, suppose that the zeros sequence of $B\in \mathcal{\widetilde{H}}_{SC}$ is fine in a strip cone $SC(1,\theta, T_1, T_2)$.
Let the strip $SL(\frac{\pi}{2}-\theta,L_{1},L_{2})$ be the image of the strip cone $SC(1,\theta,T_{1},T_{2})$ under the map $\varphi(z)$, where
\[
L_{1}:y-\tan(\frac{\pi}{2}-\theta)\cdot x-c_1=0 \ \ \  \text{and} \ \ \ L_{2}:y-\tan(\frac{\pi}{2}-\theta)\cdot x-c_2=0.
\]

Denote by $L$ the straight line passing $\varphi(z_n)$ and parallel $L_1$, and denote by $\omega_n$ the angle between $L$ and the straight line passing through $\varphi(z_{n})$ and $\varphi(z_{n+1})$. Then
\[
\sin\omega_n\le \frac{{\vert}c_1-c_2{\vert}}{{\vert}\varphi(z_{n+1})-\varphi(z_n){\vert}}.
\]
Since there are  positive numbers $C_1$ and $C_2$ such that
\begin{equation*}
0<C_{1}\le{\big \vert}\frac{1-z_{n+1}}{1-z_{n}}{\big \vert}\le C_2<1,
\end{equation*}
and
\[
\lim\limits_{n\rightarrow\infty}{\big \vert}\frac{\varphi(z_{n})}{\varphi(z_{n+1})} {\big \vert} \bigg/  {\big \vert}\frac{1-z_{n+1}}{1-z_{n}} {\big \vert}
=\lim\limits_{k\rightarrow\infty}{\big \vert}\frac{1+z_{n}}{1+z_{n+1}} {\big \vert}=1,
\]
one can see that
\begin{align*}
\lim\limits_{n\rightarrow\infty}{\vert}\varphi(z_{n+1})-\varphi(z_n){\vert}=& \lim\limits_{n\rightarrow\infty}{\vert}\varphi(z_n){\vert}\left(\frac{{\vert}\varphi(z_{n+1}){\vert}}{{\vert}\varphi(z_n){\vert}}-1\right) \\
\le& \lim\limits_{n\rightarrow\infty}\left(\frac{1}{C_2}-1\right){\vert}\varphi(z_n){\vert} \\
=&\infty.
\end{align*}
Consequently,
\[
\lim\limits_{n\rightarrow\infty}\omega_n\le \frac{\pi}{2}\lim\limits_{n\rightarrow\infty}\sin\omega_n\le \frac{\pi}{2}\lim\limits_{n\rightarrow\infty}\frac{{\vert}c_1-c_2{\vert}}{{\vert}\varphi(z_{n+1})-\varphi(z_n){\vert}}=0,
\]
that is
\[
\lim\limits_{n\rightarrow\infty}\textrm{arg}(\varphi(z_{n+1})-\varphi(z_n))=\frac{\pi}{2}-\theta.
\]

Denote by $\vartheta_n$ the angle between the imaginary axis  and  the straight line passing through $\varphi(z_{n})$ and $\varphi(z_{n+1})$.
Then, for any $\epsilon>0$, there exists a positive integer $N_2\in\mathbb{N}$ such that for every $n\ge N_2$
\[
(1-\epsilon)\sin\theta\le \sin\vartheta_n \le(1+\epsilon)\sin\theta \ \ \ \text{and} \ \ \ \textrm{Re}z_n>0.
\]

Denote by $h_n$ the distance from $\mathbf{i}y$ to the straight line passing through $\varphi(\alpha_{n}(t))$ and $\varphi(\alpha_{n}(t+\Delta t))$. Consider the area of the triangle with vertices $\mathbf{i}y$, $\varphi(\alpha_{n}(t))$ and $\varphi(\alpha_{n}(t+\Delta t))$, one can see that
\begin{align*}
&\sin\left(\textrm{arg}\frac{\varphi(\alpha_{n}(t+\Delta t))-{\bf i}y}{\varphi(\alpha_{n}(t))-{\bf i}y}\right) \\
=&\frac{{\vert}\varphi(\alpha_{n}(t+\Delta t))-\varphi(\alpha_{n}(t)){\vert}\cdot h_n}{{\vert}\varphi(\alpha_{n}(t+\Delta t))-{\bf i}y{\vert}\cdot {\vert}\varphi(\alpha_{n}(t))-{\bf i}y{\vert}} \\
\le& \frac{\frac{\textrm{Re}\varphi(\alpha_{n}(t+\Delta t))-\textrm{Re}\varphi(\alpha_{n}(t))}{(1-\epsilon)\sin\theta}\cdot \max\{{\vert}y-c_1{\vert}, {\vert}y-c_2{\vert}\}(1+\epsilon)\sin\theta}{{\vert}\varphi(\alpha_{n}(t+\Delta t))-{\bf i}y{\vert}\cdot {\vert}\varphi(\alpha_{n}(t))-{\bf i}y{\vert}} \\
=& \Delta t\cdot\frac{1+\epsilon}{1-\epsilon}\cdot \frac{(\textrm{Re}\varphi(z_{n+1})-\textrm{Re}\varphi(z_{n}))\cdot \max\{{\vert}y-c_1{\vert}, {\vert}y-c_2{\vert}\}}{{\vert}\varphi(\alpha_{n}(t+\Delta t))-{\bf i}y{\vert}\cdot {\vert}\varphi(\alpha_{n}(t))-{\bf i}y{\vert}}.
\end{align*}
Since $\{\textrm{Re}\varphi(z_{n})\}$ is an increasing sequence of positive numbers tending to $+\infty$, we have
\begin{align*}
&{\big \vert}\textrm{arg}\frac{\varphi(\alpha_{n}(t))-{\bf i}y}{\varphi(\alpha_{n}(t+\Delta t))-{\bf i}y}{\big \vert} \\
\le & \frac{\pi}{2}\sin\left(\textrm{arg}\frac{\varphi(\alpha_{n}(t+\Delta t))-{\bf i}y}{\varphi(\alpha_{n}(t))-{\bf i}y}\right)\\
\le & \Delta t\cdot\frac{\pi(1+\epsilon)}{2(1-\epsilon)}\cdot \frac{(\textrm{Re}\varphi(z_{n+1})-\textrm{Re}\varphi(z_{n}))\cdot \max\{{\vert}y-c_1{\vert}, {\vert}y-c_2{\vert}\}}{{\vert}\varphi(\alpha_{n}(t+\Delta t))-{\bf i}y{\vert}\cdot {\vert}\varphi(\alpha_{n}(t))-{\bf i}y{\vert}}.
\end{align*}
Moreover,
\begin{align*}
&{\big \vert}\textrm{arg}\frac{\varphi(\alpha_{n}(t))-{\bf i}y}{\varphi(\alpha_{n}(t+\Delta t))-{\bf i}y}{\big \vert} \\
\le & \Delta t\cdot\frac{\pi(1+\epsilon)}{2(1-\epsilon)}\cdot \frac{(\textrm{Re}\varphi(z_{n+1})-\textrm{Re}\varphi(z_{n}))\cdot \max\{{\vert}y-c_1{\vert}, {\vert}y-c_2{\vert}\}}{{\vert}\textrm{Re}\varphi(z_{n}){\vert}\cdot {\vert}\textrm{Re}\varphi(z_{n}){\vert}} \\
= & \Delta t\cdot\frac{\pi(1+\epsilon)}{2(1-\epsilon)}\cdot \left(\frac{\textrm{Re}\varphi(z_{n+1})}{\textrm{Re}\varphi(z_{n})}-1 \right)\cdot\frac{\max\{{\vert}y-c_1{\vert}, {\vert}y-c_2{\vert}\}}{\textrm{Re}\varphi(z_{n})} \\
\le & \Delta t\cdot\frac{\pi(1+\epsilon)(1-\widetilde{C_1})}{2\widetilde{C_1}(1-\epsilon)}\cdot\frac{\max\{{\vert}y-c_1{\vert}, {\vert}y-c_2{\vert}\}}{\textrm{Re}\varphi(z_{n})}.
\end{align*}

$(1)$ \ Suppose that $y$ satisfies $\min\{{\vert}y-c_1{\vert}, {\vert}y-c_2{\vert}\}\le {\vert}c_1-c_2{\vert}$.

In this case,
\[
\max\{{\vert}y-c_1{\vert}, {\vert}y-c_2{\vert}\}\le \min\{{\vert}y-c_1{\vert}, {\vert}y-c_2{\vert}\}+{\vert}c_1-c_2{\vert} \le 2{\vert}c_1-c_2{\vert}.
\]

Then,
\begin{align*}
&\sum\limits_{n=N_2}^{\infty}{\big \vert}\textrm{arg}\frac{\varphi(\alpha_{n}(t))-{\bf i}y}{\varphi(\alpha_{n}(t+\Delta t))-{\bf i}y}{\big \vert} \\
\le & \sum\limits_{n=N_2}^{\infty}\Delta t\cdot\frac{\pi(1+\epsilon)(1-\widetilde{C_1})}{2\widetilde{C_1}(1-\epsilon)}\cdot\frac{2{\vert}c_1-c_2{\vert}}{\textrm{Re}\varphi(z_{n})}\\
\le & \Delta t\cdot\frac{\pi(1+\epsilon)(1-\widetilde{C_1}){\vert}c_1-c_2{\vert}}{\widetilde{C_1}(1-\epsilon)}\cdot\sum\limits_{n=N_2}^{\infty}\frac{1}{{\vert}\varphi(z_{n}){\vert}(1-\epsilon)\sin\theta} \\
= & \Delta t\cdot\frac{\pi(1+\epsilon)(1-\widetilde{C_1}){\vert}c_1-c_2{\vert}}{\widetilde{C_1}(1-\epsilon)^2\sin\theta}\cdot\sum\limits_{n=N_2}^{\infty}\frac{{\vert}1-z_{n}{\vert}}{{\vert}1+z_{n}{\vert}} \\
\le & \Delta t\cdot\frac{\pi(1+\epsilon)(1-\widetilde{C_1}){\vert}c_1-c_2{\vert}}{\widetilde{C_1}(1-\epsilon)^2\sin\theta}\cdot\sum\limits_{n=N_2}^{\infty}{\vert}1-z_{n}{\vert}.
\end{align*}

$(2)$ \ Suppose that $y$ satisfies $\min\{{\vert}y-c_1{\vert}, {\vert}y-c_2{\vert}\}\ge {\vert}c_1-c_2{\vert}$.

Let $N\ge N_2$ be the first positive such that
\[
\textrm{Re}\varphi(z_{N})\ge \max\{{\vert}y-c_1{\vert}, {\vert}y-c_2{\vert}\}.
\]
Then,
\begin{align*}
&\sum\limits_{n=N}^{\infty}{\big \vert}\textrm{arg}\frac{\varphi(\alpha_{n}(t))-{\bf i}y}{\varphi(\alpha_{n}(t+\Delta t))-{\bf i}y}{\big \vert} \\
\le & \sum\limits_{n=N}^{\infty}\Delta t\cdot\frac{\pi(1+\epsilon)(1-\widetilde{C_1})}{\widetilde{C_1}(1-\epsilon)}\cdot\frac{\max\{{\vert}y-c_1{\vert}, {\vert}y-c_2{\vert}\}}{\textrm{Re}\varphi(z_{n})}\\
= & \Delta t\cdot\frac{\pi(1+\epsilon)(1-\widetilde{C_1})}{\widetilde{C_1}(1-\epsilon)}\cdot\frac{\max\{{\vert}y-c_1{\vert}, {\vert}y-c_2{\vert}\}}{\textrm{Re}\varphi(z_{N})} \cdot\sum\limits_{n=N}^{\infty}\frac{\textrm{Re}\varphi(z_{n})}{\textrm{Re}\varphi(z_{N})}  \\
\le & \Delta t\cdot\frac{\pi(1+\epsilon)(1-\widetilde{C_1})}{\widetilde{C_1}(1-\epsilon)} \cdot\sum\limits_{k=0}^{\infty}\widetilde{C_2}^k \\
= & \Delta t\cdot\frac{\pi(1+\epsilon)(1-\widetilde{C_1})}{\widetilde{C_1}(1-\epsilon)(1-\widetilde{C_2})}.
\end{align*}
Notice that for $t\in[0,1]$
\[
{\vert}\varphi(\alpha_{n}(t))-{\bf i}y{\vert}\ge \min\{{\vert}y-c_1{\vert}, {\vert}y-c_2{\vert}\}\sin\theta \ge {\vert}c_1-c_2{\vert}\sin\theta.
\]
It easy to see that
\begin{align*}
\frac{\max\{{\vert}y-c_1{\vert}, {\vert}y-c_2{\vert}\}}{{\vert}\varphi(\alpha_{n}(t))-{\bf i}y{\vert}}\le& \frac{\min\{{\vert}y-c_1{\vert}, {\vert}y-c_2{\vert}\}+{\vert}c_1-c_2{\vert}}{\min\{{\vert}y-c_1{\vert}, {\vert}y-c_2{\vert}\}\sin\theta} \\
=& \left(1+\frac{{\vert}c_1-c_2{\vert}}{\min\{{\vert}y-c_1{\vert}, {\vert}y-c_2{\vert}\}}\right)\cdot\frac{1}{\sin\theta} \\
\le& \frac{2}{\sin\theta}.
\end{align*}
Then, for any $N_2\le n \le N-1$,
\begin{align*}
&\sum\limits_{n=N_2}^{N-1}{\big \vert}\textrm{arg}\frac{\varphi(\alpha_{n}(t))-{\bf i}y}{\varphi(\alpha_{n}(t+\Delta t))-{\bf i}y}{\big \vert} \\
\le & \sum\limits_{n=N_2}^{N-1}\Delta t\cdot\frac{\pi(1+\epsilon)}{2(1-\epsilon)}\cdot \frac{(\textrm{Re}\varphi(z_{n+1})-\textrm{Re}\varphi(z_{n}))\cdot \max\{{\vert}y-c_1{\vert}, {\vert}y-c_2{\vert}\}}{{\vert}\varphi(\alpha_{n}(t+\Delta t))-{\bf i}y{\vert}\cdot {\vert}\varphi(\alpha_{n}(t))-{\bf i}y{\vert}}\\
\le & \Delta t\cdot\frac{\pi(1+\epsilon)}{2(1-\epsilon)}\cdot \sum\limits_{n=N_2}^{N-1}\frac{(\textrm{Re}\varphi(z_{n+1})-\textrm{Re}\varphi(z_{n}))\cdot \max\{{\vert}y-c_1{\vert}, {\vert}y-c_2{\vert}\}}{\min\{{\vert}y-c_1{\vert}, {\vert}y-c_2{\vert}\}\sin\theta\cdot {\vert}\varphi(\alpha_{n}(t))-{\bf i}y{\vert}}\\
\le & \Delta t\cdot\frac{\pi(1+\epsilon)}{(1-\epsilon)\min\{{\vert}y-c_1{\vert}, {\vert}y-c_2{\vert}\}\sin^2\theta}\cdot \sum\limits_{n=N_2}^{N-1}(\textrm{Re}\varphi(z_{n+1})-\textrm{Re}\varphi(z_{n}))\\
=& \Delta t\cdot\frac{\pi(1+\epsilon)}{(1-\epsilon)\sin^2\theta}\cdot \frac{\textrm{Re}\varphi(z_{N})-\textrm{Re}\varphi(z_{N_2})}{\min\{{\vert}y-c_1{\vert}, {\vert}y-c_2{\vert}\}}\\
\le& \Delta t\cdot\frac{\pi(1+\epsilon)}{(1-\epsilon)\sin^2\theta}\cdot \frac{\frac{\textrm{Re}\varphi(z_{N})}{\textrm{Re}\varphi(z_{N-1})}\cdot \textrm{Re}\varphi(z_{N-1})}{\min\{{\vert}y-c_1{\vert}, {\vert}y-c_2{\vert}\}}\\
\le& \Delta t\cdot\frac{\pi(1+\epsilon)}{(1-\epsilon)\sin^2\theta}\cdot \frac{\frac{1}{\widetilde{C_1}}\max\{{\vert}y-c_1{\vert}, {\vert}y-c_2{\vert}\}}{\min\{{\vert}y-c_1{\vert}, {\vert}y-c_2{\vert}\}} \\
=& \Delta t\cdot\frac{2\pi(1+\epsilon)}{(1-\epsilon)\widetilde{C_1}\sin^2\theta}.
\end{align*}
Thus, in this case,
\begin{align*}
&\sum\limits_{n=N_2}^{\infty}{\big \vert}\textrm{arg}\frac{\varphi(\alpha_{n}(t))-{\bf i}y}{\varphi(\alpha_{n}(t+\Delta t))-{\bf i}y}{\big \vert} \\
\le& \Delta t\cdot\left(\frac{2\pi(1+\epsilon)}{(1-\epsilon)\widetilde{C_1}\sin^2\theta}+\frac{\pi(1+\epsilon)(1-\widetilde{C_1})}{\widetilde{C_1}(1-\epsilon)(1-\widetilde{C_2})}\right).
\end{align*}
Therefore, we could write
\[
K_2=\left\{\begin{matrix}
\frac{\pi(1+\epsilon)(1-\widetilde{C_1}){\vert}c_1-c_2{\vert}}{\widetilde{C_1}(1-\epsilon)^2\sin\theta}\cdot\sum\limits_{n=N_2}^{\infty}{\vert}1-z_{n}{\vert}, \\ \frac{2\pi(1+\epsilon)}{(1-\epsilon)\widetilde{C_1}\sin^2\theta}+\frac{\pi(1+\epsilon)(1-\widetilde{C_1})}{\widetilde{C_1}(1-\epsilon)(1-\widetilde{C_2})}
\end{matrix}
\right\}
\]
as required.
\end{proof}

\begin{proposition}\label{prop3}
For any $B\in \mathcal{\widetilde{H}}_{SC}$, there exists a positive integer $N_3\in\mathbb{N}$ and a positive number $K_3$ such that for any $t\in [0,1]$ and small positive number $\Delta t$,
\[
\sum_{n}{\big \vert}\textrm{arg}\frac{1-\alpha_{n}(t)}{1-\alpha_{n}(t+\Delta t)}{\big \vert}\le K_3\Delta t.
\]
\end{proposition}
\begin{proof}
Without loss of generality, suppose that the zeros sequence of $B\in \mathcal{\widetilde{H}}_{SC}$ is fine in a strip cone $SC(1,\theta, T_1, T_2)$.
Let the strip $SL(\frac{\pi}{2}-\theta,L_{1},L_{2})$ be the image of $SC(1,\theta,T_{1},T_{2})$ under the map $\varphi(z)$, where
\[
L_{1}:y-\tan(\frac{\pi}{2}-\theta)\cdot x-c_1=0 \ \ \  \text{and} \ \ \ L_{2}:y-\tan(\frac{\pi}{2}-\theta)\cdot x-c_2=0.
\]
For convenience, assume the line $L_2$ is on the right of the line $L_1$. By translating $L_{2}$ one unit to the right, we obtain a new line $\widehat{L_2}$, more precisely,
\[
\widehat{L_{2}}: \ y-\tan(\frac{\pi}{2}-\theta)\cdot (x-1)-c_2=0.
\]
Let
\[
\varphi(\widehat{z_n})=\varphi(z_{n})+1 \ \ \text{and} \ \ \varphi(\widehat{\alpha}_{n}(t))=\varphi(\alpha_{n}(t))+1.
\]
It is not difficult to see that $\{\varphi(\widehat{z_n})\}_{n=1}^{\infty}$ also satisfies the following conditions as well as a fine sequence.
\begin{enumerate}
\item[(1)] \ The sequence $\{\varphi(\widehat{z_n})\}_{n=1}^{\infty}$ lies in the strip $SL(\frac{\pi}{2}-\theta,L_{1},\widehat{L_{2}})$.
\item[(2)] \ $\lim\limits_{n\rightarrow\infty}\textrm{arg}(\varphi(z_{n+1})-\varphi(z_n))=\frac{\pi}{2}-\theta$.
\item[(3)] \ ${\rm Re} \varphi(\widehat{z_n})$ monotonically tends to $+\infty$, and there are two positive numbers $\widetilde{D_1}$ and $\widetilde{D_2}$ such that
\[
0<\widetilde{D_1}\leq \frac{{\rm Re}\varphi(\widehat{z_n})}{{\rm Re}\varphi(\widehat{z_{n+1}})} \leq \widetilde{D_2}<1.
\]
\end{enumerate}
Then, the Proposition \ref{prop2} also holds for the sequence $\{\varphi(\widehat{z_n})\}_{n=1}^{\infty}$. That is,  there exists a positive integer $N_3\in\mathbb{N}$ and a positive number $K_3$ such that for any $t\in [0,1]$ and small positive number $\Delta t$,
\[
\sup\limits_{y\in \mathbb{R}}\textrm{ess}\sum\limits_{n=N_3}^{\infty}{\big \vert}\textrm{arg}\frac{\varphi(\widehat{\alpha}_{n}(t))-{\bf i}y}{\varphi(\widehat{\alpha}_{n}(t+\Delta t))-{\bf i}y}{\big \vert}\le K_3\Delta t.
\]
In particular, the above inequality holds for $y=0$, and hence
\begin{equation*}
\begin{aligned}
K_3\Delta t \ge & \sum\limits_{n=N_3}^{\infty}{\big \vert}\textrm{arg}\frac{\varphi(\widehat{\alpha}_{n}(t))}{\varphi(\widehat{\alpha}_{n}(t+\Delta t))}{\big \vert} \\
=&\sum\limits_{n=N_3}^{\infty}{\big \vert}\textrm{arg}\frac{\varphi(\alpha_{n}(t))+1}{\varphi(\alpha_{n}(t+\Delta t))+1}{\big \vert}\\
=&\sum\limits_{n=N_3}^{\infty}{\big \vert}\textrm{arg}\frac{\frac{1+\alpha_{n}(t)}{1-\alpha_{n}(t)}+1}{\frac{1+\alpha_{n}(t+\Delta t)}{1-\alpha_{n}(t+\Delta t)}+1}{\big \vert}\\
=&\sum\limits_{n=N_3}^{\infty}{\big \vert}\textrm{arg}\frac{1-\alpha_{n}(t+\Delta t)}{1-\alpha_{n}(t)}{\big \vert}.
\end{aligned}
\end{equation*}
This completes the proof.
\end{proof}

{\bf Proof  of  Main Theorem :}  Given any $B\in\mathcal{{H}}_{SC}$. Without loss of generality, we may assume that its zeros lie in a strip cone $SC(1,\theta,T_{1},T_{2})$, $\theta\in(0, \pi)$. By Lemma \ref{finesub}, it has a factor
$\widetilde{B}\in\mathcal{\widetilde{H}_{SC}}$, denoted by
\[
\widetilde{B}(z)=\prod\limits_{n=1}^{\infty}\frac{{\vert}z_{n}{\vert}}{z_{n}}\frac{z_{n}-z}{1-\overline{z_{n}}z}.
\]
Following from Proposition \ref{prop1}, Proposition \ref{prop2}, Proposition \ref{prop3} and Lemma \ref{Nest}, we could choose $N\ge \max\{N_1, N_2, N_3\}$ and then there is a continuous path from
\[
\prod\limits_{n=N}^{\infty}\frac{{\vert}z_{n}{\vert}}{z_{n}}\frac{z_{n}-z}{1-\overline{z_{n}}z}\ \ \ \text{to} \ \ \ \prod\limits_{n=N+1}^{\infty}\frac{{\vert}z_{n}{\vert}}{z_{n}}\frac{z_{n}-z}{1-\overline{z_{n}}z}.
\]
Therefore, by the path-connectedness of M\"{o}bius transformations and Lemma \ref{factorconnect},
we have $B\sim zB$.

\section*{Declarations}
\begin{itemize}
\item Ethics approval

\noindent Not applicable.

\item Competing interests

\noindent The author declares that there are no conflict of interest or competing interests.

\item Authors' contributions

\noindent All authors reviewed this paper.

\item Funding

\noindent There is no funding source for this manuscript.

\item Availability of data and materials

\noindent Data sharing is not applicable to this article as no datasets were generated or analyzed during the current study.
\end{itemize}

\end{document}